\documentclass[12pt, a4paper]{article}

\usepackage{pdfsync}
\usepackage{amsmath}\multlinegap=0pt
\usepackage[latin1]{inputenc}
\usepackage{amsthm, amsmath}
\usepackage{amssymb}
\usepackage{mathtools,thmtools,stmaryrd,esint}
\usepackage{calrsfs}
\usepackage{hyperref}
\usepackage{graphicx}

\usepackage{capt-of} 
\usepackage{comment}
\usepackage{color}
\usepackage{tikz}
\definecolor{orange1}{rgb}{1,0.5,0}
\definecolor{orange2}{rgb}{0.8,0.4,0}
\definecolor{orange3}{rgb}{0.6,0.3,0}

\definecolor{vert1}{rgb}{0,1,0}
\definecolor{vert2}{rgb}{0,0.8,0}
\definecolor{vert3}{rgb}{0,0.7,0}

\definecolor{bleu1}{rgb}{0.8,0,1}
\definecolor{bleu2}{rgb}{0.64,0,0.8}
\definecolor{bleu3}{rgb}{0.56,0,0.7}

\definecolor{gris1}{rgb}{0.7,0.7,0.7}
\definecolor{gris2}{rgb}{0.6,0.6,0.6}
\definecolor{gris3}{rgb}{0.55,0.55,0.55}

\definecolor{bleugris}{rgb}{0.48,0.62,0.83}

\numberwithin{equation}{section}
\theoremstyle{plain}
\newtheorem{thm}{Theorem}[section]

\newtheorem{prop}[thm]{Proposition}
\newtheorem{lem}[thm]{Lemma}
\newtheorem{defi}[thm]{Definition}

\theoremstyle{definition}

\theoremstyle{remark}

\newcommand{\R}{\mathbb{R}}

\newcommand\N{{\mathbb N}}

\newcommand\abs[1]{\left\vert#1\right\vert}
\newcommand\norm[1]{\left\Vert#1\right\Vert}

\newcommand\pref[1]{(\ref{#1})}

\let \eps\varepsilon

\DeclareMathOperator{\argmin}{argmin}

\DeclareMathOperator{\argmax}{argmax}

\def\<#1,#2>{\left<#1,#2\right>}

\newcommand\tilw{\widetilde {w}}

\def\PP{{\cal P}}

\newcommand{\cU}{\mathcal{U}}
\newcommand{\EE}{\mathbb{E}}
\newcommand{\Pro}{\mathbb{P}}
\def\ow{\overline{w}}
\def\uw{\underline{w}}
\def\wmu{\widetilde{\mu}}

\def\tmu{\widetilde{\mu}}
\def\tJ{\widetilde{J}}
\def\tpi{\widetilde{\pi}}
\def\tw{\widetilde{w}}

\def\tf{\widetilde{f}}
\def\tR{\widetilde{R}}
\def\tG{\widetilde{G}}
\def\tR{\widetilde{R}}
\def\txi{\widetilde{\xi}}
\def\tR{\widetilde{R}}
\def\tA{\widetilde{A}}

\title{A simple city equilibrium model with an application to teleworking}

\author{Yves Achdou\thanks{Universit\'e de Paris Cit\'e and Sorbonne Universit\'e, CNRS, Laboratoire Jacques-Louis Lions, (LJLL), F- 75006 Paris, France
\texttt{achdou@ljll-univ-paris-diderot.fr}},
\and
Guillaume Carlier\thanks{CEREMADE, Universit\'e Paris
Dauphine, PSL, Pl. de Lattre de Tassigny, 75775 Paris Cedex 16, FRANCE and INRIA-Paris, MOKAPLAN,
\texttt{carlier@ceremade.dauphine.fr}},
\and
Quentin Petit \thanks{CEREMADE, Universit\'e Paris
Dauphine, PSL, and EDF R\&D,
\texttt{quentin.petit@edf.fr}}
\and
Daniela Tonon \thanks{Dipartimento di Matematica "Tullio Levi-Civita", Universit\`a degli Studi di Padova, via Trieste 63, 35121 Padova, Italy.
\texttt{tonon@math.unipd.it}}
}

\begin{document}

\maketitle

\begin{abstract}
We propose a simple semi-discrete spatial  model where rents, wages and the density of  population in a city can be deduced from free-mobility and equilibrium conditions on the labour and residential housing markets. We prove existence and (under stronger assumptions) uniqueness of the equilibrium. We extend our model to the case where teleworking is introduced.  We present numerical simulations which shed light on the effect of teleworking on the structure of the city at equilibrium. 
\end{abstract}

\textbf{Keywords:} spatial equilibria, teleworking.

\smallskip

\textbf{MS Classification: 90B85, 49Q22.}

\section{Introduction}\label{sec-intro}

Whether the internal structure of cities can be determined endogeneously by equilibrium (e.g. on the labour and residential markets) conditions is a central and generally complex issue in urban economics, see in particular  Fujita and Ogawa \cite{FUJITAOGAWA}, Lucas and Rossi-Hansberg \cite{LRH2002} and the references therein. Due to the fact that equilibria are in general non-unique, difficult to compute explicitly and may depend quite crucially on some modelling assumption on the city itself (linear, circular, finite etc...), it is in general difficult to anticipate how changes in some parameters  influence the city structure and how rents and spatial distributions of agents evolve under such changes. With the framework proposed in \cite{CE2007} and  relying  on optimal transportation, existence of equilibria for arbitrary city shapes can be proved, but the resulting model is far too complex to yield meaningful predictions on how spatial equilibria evolve under shocks on the technology used by the firms or on any other fundamentals in the city economy.

\smallskip

The tremendous development of telecommuting and the possibility of working from home (anticipated and analyzed in the seminal paper of Gaspar and Glaeser \cite{GASPAR} in 1998), exacerbated since the beginning of the COVID-19 pandemic in 2020, has seriously challenged the traditional paradigm where the commuting cost to business districts is a key ingredient in workers' housing choice which shapes the rents and residential distribution patterns. This has of course stimulated an active stream of recent research. Delventhal et al. \cite{DELVENTHAL} (also see \cite{DelventhalParkhomenko_WFH}) developed a model to analyze the impact of the increase of teleworking in the Los Angeles metropolitan area; among their main findings:  \emph{residents move to the periphery} and \emph{average real estate prices fall, with declines in core locations and increases in the periphery}. In the recent equilibrium model of Behrens et al. \cite{Behrens}, existence and uniqueness of equilibria is established as well as the (mixed) effect of working from home on efficiency.

\smallskip

In the present paper, our goal is to present a simple spatial equilibrium model which is tractable enough to compute numerically equilibria in one or two-dimensional domains and analyze the effect of the increase of working from home. Our model is inspired by discrete choice models and entropic optimal transport (see  \cite{CJP2009}, \cite{CM2018} and the book \cite{Galichonbook} for an overview of these methods in economics). We consider that firms are located at finitely many fixed places and that workers are identical and distributed continuously throughout the city. Wages determine labour supply for the different firms through some simple Gibbs distributions, equilibrium on the housing market gives the residential distribution in terms of workers' revenues. Equilibrium wages are then obtained by clearing the labour market so that our model determines wages, rents and the spatial distribution of (possibly working from home) workers.  On the theoretical side, we provide a simple existence argument and identify conditions which guarantee uniqueness.  As we will show, our model is well suited to capture the introduction of teleworking. On the applied side, we can numerically compute equilibria, perform some comparative statics  as working from home increases; our findings are in line with those of \cite{DELVENTHAL}.  Note however that our model differs from those of \cite{DELVENTHAL} and \cite{Behrens}  in two ways: on the one hand, it is less general because we consider identical workers (that may work from home or on site) and do not distinguish between skilled (allowed to work from home) and unskilled (having to commute) workers; but, on the other hand, it allows for general commuting costs, city shapes and a continuum of housing locations.

\smallskip

The paper is organized as follows. The model is introduced and equilibria are defined in section \ref{sec-model}. In section \ref{sec-exist}, we establish existence of equilibria and also prove a uniqueness result under additional assumptions. Section \ref{sec-tele} is devoted to a modification of the model which takes teleworking into account. Numerical simulations are presented in section \ref{sec-num}. Some proofs are gathered in the appendix.

\section{The model}\label{sec-model}

\subsection{Setting and assumptions}\label{assum}

The city is modeled as a compact subset $X$ of $\R^d$ with positive $d$-dimensional Lebesgue measure. In this city, there are firms and workers. Our aim is to describe a tractable spatial equilibrium model for  the labour and housing markets. Our model is of semi-discrete type: the locations of firms are fixed and discrete whereas workers occupy the whole city according to a certain (absolutely continuous) distribution $\mu$ which is determined by equilibrium conditions. Workers have to commute and will choose to work at locations for which their revenue (i.e. wage net of commuting cost) is maximal. We also assume that the total size of the population is fixed which (up to normalization) amounts to saying that $\mu$ belongs to $\PP(X)$, the set of Borel probability measures on $X$. The fact that the  spatial distribution of firms is discrete and that the distribution of workers is absolutely continuous implies that firms do not consume land whereas workers do. Therefore, in this semi-discrete setting, there is no competition for land use between firms and workers; this simplifies a lot the equilibirum condition  (e.g. compared to the models of \cite{LRH2002} or \cite{CE2007}) on the land market which is reduced to the residential housing market. We now detail the specifications of the model.

\smallskip

{\textbf{Firms}.} There are finitely many firms which are indexed by $i=1, \ldots, N$, firm $i$ is located at a point of $X$ denoted $y_i$. Firms locations $\{y_1, \ldots, y_N\}$ are fixed. Firm $i$ hiring a quantity of labour $l_i$ produces a quantity $f_i(l_i)$. We shall always assume that the production functions $f_i$: $\R_+ \to \R_+$ are increasing, strictly concave, continuous on $\R_+$ and $C^1$ on $(0,+\infty)$ with
\begin{equation}\label{assumptfi}
f_i(0)=0, \; \lim_{l\to +\infty} f_i(l)=+\infty,   \; \lim_{l\to 0^+} f_i'(l)=+\infty, \; \lim_{l\to +\infty} f_i'(l)=0. 
\end{equation}
If the wage paid by firm $i$ is denoted by $w\in \R_+$, its employment level is obtained by maximizing profit:
\begin{equation}
\pi_i(w):=\sup_{l\geq 0} \{ f_i(l)-w l\}.
\end{equation} 
Obviously, $\pi_i(0)=+\infty$, $\pi_i$ is strictly convex, nonnegative and nonincreasing and thanks to \pref{assumptfi}  and the envelope theorem, it is also $C^1$ on $(0, +\infty)$, and for every $w>0$, firm $i$'s labour demand is given by
\begin{equation}\label{Lienv}
L_i(w):=\argmax_{l\geq 0} \{ f_i(l)-w l\}=-\pi'_i(w).
\end{equation}
Moreover, one readily deduces from  \pref{assumptfi}  that
\begin{equation}\label{prilimit}
\pi_i(0)=\lim_{w\to 0^+} \pi_i(w)=+\infty, \;  \;  \lim_{w\to +\infty} \pi_i(w)=0,
\end{equation}
and
\begin{equation}\label{Llimit}
\lim_{w\to 0^+} L_i(w)=+\infty, \; \lim_{w\to +\infty} L_i(w)=0. 
\end{equation}

\smallskip

{\textbf{Workers preferences, deducing density from revenues}.} Workers spend their revenue in consumption $C$ on the one hand and renting their house on the other hand; the surface they rent is denoted $S$. Workers are assumed to all have the same Cobb-Douglas utility:
\begin{equation}\label{cobbdu}
(C,S)\in \R_+^2 \mapsto C^{\theta} S^{1-\theta}
\end{equation}
for some $\theta\in [0,1)$. Normalizing the price of the consumption good to $1$, workers with revenue $R>0$, facing a rent $Q>0$, solve
\begin{equation}
	\cU_\theta(R,Q)=\sup\left\{C^\theta S^{1-\theta} \; : \; C + Q S\le R, C\ge 0, S\ge 0\right\},
\end{equation}
the solution being
\begin{equation}\label{formCS}
C= \theta R,\quad\quad\hbox{and}\quad\quad S = (1-\theta)\frac{R}{Q}.
\end{equation}
One gets the closed-form expression
\begin{equation}\label{formcU}
	\cU_\theta(R,Q) = \theta^\theta(1-\theta)^{1-\theta} \frac{R}{Q^{1-\theta}} \mbox{ if $\theta\in (0,1)$ and } \cU_0(R,Q)=\frac{R}{Q}.
\end{equation}

Revenues, surface consumption and rents are functions of the workers' residential location $x\in X$, which we now denote $R(x)$, $S(x)$ and $Q(x)$. If we denote by $\mu(x)$ the density of the workers' residential distribution and if we normalize 
the supply of land to be uniform\footnote{Of course, it is also possible to have a location dependent supply for land and to replace \pref{surface1} by $\mu(x) S(x)= \alpha(x)$ for a given nonnegative function $\alpha$.}, $\mu$ and $S$ are simply related by
\begin{equation}\label{surface1}
\mu(x) S(x)=1, \; \forall x\in X.
\end{equation}
In particular the support of $\mu$ is the whole city $X$. Since workers are identical and free to choose where to live, at equilibrium, utility should be constant which, in view of \pref{formcU}, yields
\[Q \propto R^{\frac{1}{1-\theta}},\]
where we write $a\propto b$ to indicate the functions $a$ and $b$ are proportional. With \pref{formCS} and \pref{surface1}, we deduce
\[\mu \propto \frac{Q}{R} \propto R^{\frac{\theta}{1-\theta}}.\]
Since the total mass of $\mu$ has been fixed to $1$, we deduce an explicit dependence between density and revenue:
\begin{equation}\label{mudeR}
\mu(x)=\frac{ R(x)^{\frac{\theta}{1-\theta}}   } {\int_X R(y)^{\frac{\theta}{1-\theta}} \mbox{d} y   }, \; \forall x\in X.
\end{equation}

\smallskip

{\textbf{Free-mobility of labour, commuting costs, deducing revenues from wages.}} Let us denote  by $w:=(w_1, \ldots, w_N) \in (0,+\infty)^N$ the wages paid to workers hired by firm $i=1, \ldots, N$. If agents living at $x$ choose to work for firm $i$ located at $y_i$, they incur a commuting cost $c_i(x)$ so that their revenue is $w_i-c_i(x)$. We assume that the commuting (or accessibility) costs are continuous functions:
\begin{equation*}
c_i\in C(X), \; \forall i\in \{1, \ldots, N\}. 
\end{equation*}
We finally assume that agents have the possibility to stay at home ($c_0 \equiv 0$) and receive a given exogeneous wage $w_0>0$. In the absence of random shocks on revenues, the revenue of agents located at $x$, is given by
\begin{equation}\label{defr0}
R_0(x,w):=\max_{i=0, \ldots, N} (w_i-c_i(x))
\end{equation}
which is a convex function of wages but may have kinks. For mathematical simplicity, we will start by replacing it by its softmax (smooth) version:
\begin{equation}\label{defrsig}
R_\sigma(x,w):= \sigma \log \Big( \sum_{i=0}^N  e^{\frac{w_i-c_i(x)}{\sigma}}\Big)
\end{equation}
where $\sigma>0$ is a regularization parameter, and we postpone the analysis of the limiting case $\sigma=0$ to paragraph \ref{nonoise}.  Note that $R_\sigma(x,.)$ is smooth and strictly convex and that it is an approximation of $R_0$ when $\sigma$ is small because of the obvious inequality:
\begin{equation}\label{softmaxapprox} 
R_0(x,w) \leq R_\sigma(x,w) \leq \sigma \log(N+1)+ R_0(x, w), \forall (x,w)\in X\times \R^N.
\end{equation} 

The regularized form \pref{defrsig} can also be justified by  considering independent random shocks to the revenue, indeed, one has
\begin{equation}\label{softmaxmax}
R_\sigma(x,w):=\EE \Big( \max_{i=0, \ldots, N} (w_i-c_i(x)+\sigma \eps_i)\Big)
\end{equation}
where $\eps_0, \ldots, \eps_N$ are i.i.d. centered Gumbel random variables (see Appendix A for details). Given the wages $w$, the probability that an agent living at $x$ works for firm $i$, is given by
\[ \Pro( x \mbox{ works at $i$})=\Pro \Big(w_i-c_i(x) + \sigma \eps_i=  \max_{j=0, \ldots, N} (w_j-c_j(x)+\sigma \eps_j)\Big);\]
this probability is given by the Gibbs distribution (see Appendix A):
\begin{equation}\label{probaix}
\Pro( x \mbox{ works at $i$})
=\frac{\partial R_\sigma}{\partial w_i}(x,w)
= \frac{e^\frac{w_i-c_i(x)}{\sigma}}{\sum_{j=0}^{N} e^\frac{w_j-c_j(x)}{\sigma}}.
\end{equation}

Hence the total labour supply for firm $i$ induced by the vector of wages $w$ is
\begin{equation}\label{laboursupply}
\int_X \frac{\partial R_\sigma}{\partial w_i}(x,w) \mu(x) \mbox{d}x=  \int_X \frac{e^\frac{w_i-c_i(x)}{\sigma}}{\sum_{j=0}^{N} e^\frac{w_j-c_j(x)}{\sigma}}  \mu(x) \mbox{d}x. 
\end{equation}

\subsection{Equilibria}

A spatial equilibrium is a configuration that clears both the labour and residential housing markets. It consists of a vector of wages $w=(w_1, \ldots w_N)$ and a probability density $\mu$, such that for each $i$, labour demand given by \pref{Lienv} matches labour supply given by \pref{laboursupply}:
\begin{equation}\label{eqlabour}
\pi'_i(w_i)+ \int_X \frac{\partial R_\sigma}{\partial w_i}(x,w) \mu(x) \mbox{d}x=0, \; \forall i\in \{1, \ldots, N\}.
\end{equation}
But at the same time, by free-mobility of labour, revenues of workers can be deduced from wages by formula \pref{defrsig}, and equilibrium on the housing market requires utility to be constant so that $\mu$ can be deduced from $R(.)=R_\sigma(.,w)$ by formula \pref{mudeR} i.e.
\begin{equation}\label{defmuw}
\mu(x)= \mu_w(x):= \frac{ R_\sigma(x, w)^{\frac{\theta}{1-\theta}}   } {\int_X R_\sigma(y, w)^{\frac{\theta}{1-\theta}} \mbox{d} y   }, \; \forall x\in X.
\end{equation}
This leads to the following definition:

\begin{defi}\label{equidef}
An equilibrium is a vector of wages $(w_1, \ldots, w_N)\in (0,+\infty)^N$ such that the system of $N$ equations \pref{eqlabour} is satisfied for the probability density $\mu=\mu_w$ given by \pref{defmuw}.
\end{defi}

\section{Existence and uniqueness}\label{sec-exist}

\subsection{Existence of equilibria}

To prove existence of equilibria, it will be convenient to observe that for a given $\mu\in \PP(X)$, the system \pref{eqlabour} is the first-order optimality condition equation for the convex minimization problem:
\begin{equation}\label{minijmu}
\inf_{w\in \R_+^N}  J_\mu(w) \mbox{ where } J_\mu(w)  :=\sum_{i=1}^N \pi_i(w_i) + \int_X R_{\sigma}(x, w) \mbox{d} \mu(x).
\end{equation}

Before going further, we would like to remark that  the dual formulation of \pref{minijmu} is naturally related to entropic optimal transport. Indeed, the Fenchel-Rockafellar dual problem of \pref{minijmu} reads as the maximization problem over labour variables:
\[\sup_{(l_1, \ldots, l_N)\in \R_+^N} \left\{\sum_{i=1}^N f_i(l_i) -C_{\sigma}(l)\right\}  \]
where 
\begin{equation}\label{ctlr}
C_\sigma(l):=\sup_{w\in \R_+^N} \Big\{\sum_{i=1}^N l_i w_i -\int_X R_{\sigma}(x,w) \mbox{d} \mu(x) \Big\}
\end{equation}
it is obvious that $C_{\sigma}(l)=+\infty$ unless $l_i\geq 0$ and $\sum_{i=1}^{N} l_i \leq 1$. For such $(l_1, \ldots, l_N)$ setting $l_0:=1-\sum_{i=1}^N l_i$, we may view $l=(l_0, l_1, \ldots, l_N)$ as a probability vector with $l_0$ being the probability of not working  (for wage $w_0$ and zero commuting cost) and $l_i$, $i\geq 1$, being the probability of working for firm $i$; in this case, the optimality conditions for the optimal $w$ in \pref{ctlr} are
\[l_i=\int_X \frac{e^{\frac{w_i-c_i(x)}{\sigma}}}{ \sum_{j=0}^N e^{\frac{w_j-c_j(x)}{\sigma}}} \mbox{d}\mu(x), \; i=0, \ldots, N\]
which is easily seen to imply that the probabilities
\[\gamma(i,x):=\frac{e^{\frac{w_i-c_i(x)}{\sigma}}}{ \sum_{j=0}^N e^{\frac{w_j-c_j(x)}{\sigma}}}\]
solve the entropic optimal transport problem
\[\inf_\gamma \Big\{ \sum_{i=0}^N c_i(x) \gamma(i,x) \mbox{d} \mu(x) + \sigma \sum_{i=0}^N \int_X \gamma(i,x) \log(\gamma(i,x)) \mbox{d} \mu(x)\Big\}\]
subject to the mass conservation constraints:
\[\sum_{i=0}^N \gamma(i,x)=1, \; \forall x, \; l_i=\int_X \gamma(i,x) \mbox{d} \mu(x), \; i=0, \ldots, N.\]
In other words, the dual of \pref{minijmu} consists in maximizing the total production net of (entropically regularized) transport. For more on connections between semi-discrete spatial equilibria and optimal transport, we refer to Crippa, Jimenez and Pratelli \cite{CJP2009} and Carlier and Mallozzi \cite{CM2018}.

\begin{lem}\label{variat}
For every $\mu\in \PP(X)$, \pref{minijmu} admits a unique minimizer  $w^*(\mu)$ and there exist  constants $\uw$ and $\ow$  that do not depend on $\mu$  such that  $0< \uw<\ow$ and $w^*(\mu)\in [\uw, \ow]^N$. Moreover $w^*(\mu)$ is the only solution of \pref{eqlabour} and  the map $\mu \in \PP(X) \mapsto w^*(\mu)\in [\uw, \ow]^N$ is weakly $*$ continuous. 
\end{lem}

\begin{proof}
Let $M:= \max_i \Vert c_i \Vert_{\infty}$; due to the form of $R_\sigma$ in \pref{defrsig}, one has
\[ \sum_{i=1}^N \pi_i(w_i)+ \sigma  \log(N+1)+  \max_{i=0, \ldots, N} w_i +M   \geq J_\mu(w) \geq \sum_{i=1}^N \pi_i(w_i) + \max_{i=0, \ldots, N} w_i -M.\]
Let $w\in \R_+^N$ be such that  $J_\mu(w) \leq J_\mu(w_0, \ldots, w_0)$, then, one has:
\[\begin{split}
 \max_{i=1, \ldots, N} w_i  + \sum_{i=1}^N \pi_i(w_i)& \leq M + J_\mu(w_0, \ldots w_0) \\
&\leq  \ow:= 2M + \sum_{i=1}^N \pi_i(w_0) + \sigma \log(N+1)+ w_0
\end{split}\]
since, for every $i$, both $w_i$ and $ \pi_i(w_i)$ are nonnegative, this yields
\[\max_{i=1, \ldots, N} w_i \leq \ow, \mbox{ and } \max_{i=1, \ldots, N} \pi_i(w_i) \leq  \sum_{i=1}^N \pi_i(w_i) \leq \ow  \]
and thanks to \pref{prilimit}, there exists $\uw\in (0, \ow)$ such that $\pi_i^{-1}([0, \ow]) \subset  [\uw, +\infty)$. Hence the infimum in \pref{minijmu} coincides with $\inf_{w \in [\uw, \ow]^N} J_\mu(w)$ which is achieved by  continuity of $J_\mu$ on the compact set  $[\uw, \ow]^N$. Uniqueness follows from the strict convexity of $J_\mu$ and since the minimizer of $J_\mu$, $w^*(\mu)$,  lies in  interior of $\R_+^N$, it is the unique critical point of $J_\mu$ hence the only solution of \pref{eqlabour}. Finally, let us assume that $(\mu_n)_n$ is a sequence in $\PP(X)$ weakly $*$ converging to $\mu$, then $w_n:= w^*(\mu_n)$ taking values in the compact set $[\uw, \ow]^N$, $(w_n)_n$ admits a (not relabeled) subsequence which converges to some $w^*\in [\uw, \ow]^N$. Since $R_\sigma(., w_n)$ converges uniformly to $R_\sigma(., w)$, we have
\[ \lim_n J_{\mu_n}(w_n) = J_\mu(w^*)\]
and since for every $w\in (0, +\infty)^N$, $J_{\mu_n}(w_n) \leq J_{\mu_n}(w)$, passing to the limit yields
\[J_{\mu}(w^*)\leq J_\mu(w)\]
so that $w^*=w^*(\mu)$ and the whole sequence $(w_n)_n$ converges to $w^*(\mu)$. 

\end{proof}

\begin{thm}\label{existeqsig}
Under the general assumptions of paragraph \ref{assum}, there exists an equilibrium (in the sense of definition \ref{equidef}) which corresponds  to a vector of wages $w\in [\uw, \ow]^N$ where $0< \uw<\ow$ are the  bounds from Lemma \ref{variat}.
\end{thm}

\begin{proof}
Let us define for every $w\in [\uw, \ow]^N$, $\Phi(w):=w^*(\mu_w)$ where $w^*$ is the continuous map defined in Lemma \ref{variat} and $\mu_w$ is defined by \pref{defmuw}. It follows from Lemma \ref{variat} that $\Phi$ is a continuous self-map of  $[\uw, \ow]^N$, it therefore admits a fixed-point thanks to Brouwer's fixed-point theorem, such a fixed-point being an equilibrium by construction, this ends the proof.

\end{proof}

\subsection{A regime of uniqueness}

We are now going to combine the structure of the equilibrium conditions, the implicit function theorem and a continuation argument, to establish uniqueness of the equilibrium if the exponent $\theta$  is small enough (see \pref{cobbdu} for the Cobb-Douglas specification of agents preferences). Since we will need to differentiate the equilibrium conditions, we shall need an extra degree of smoothness of the production functions:
\begin{equation}\label{fisstrict}
\forall i \in \{1, \ldots, N\}, \; f_i \in C^2((0, + \infty)) \mbox{ and }  \; f''_i(l)<0, \; \forall l\in (0,+\infty).
\end{equation}
Since $-\pi'_i$ is the inverse of $f'_i$, this implies 
\begin{equation}\label{pistrict}
\forall i \in \{1, \ldots, N\}, \; \pi_i \in C^2((0, + \infty)) \mbox{ and }  \; \pi''_i(w)>0, \; \forall w\in (0,+\infty),
\end{equation}
so that, for every $\mu\in \PP(X)$, the function $J_\mu$ is $C^2$ and strongly convex on $[\uw, \ow]^N$. The basic idea behind the proof is easy to grasp: setting $\alpha:=\frac{\theta}{1-\theta}$, recalling \pref{defmuw} and defining
\begin{equation}\label{deftilmu}
\wmu(x, w, \alpha):=  \frac{ R_\sigma(x, w)^{\alpha}} {\int_X R_\sigma(y, w)^{\alpha} \mbox{d} y }, \; \forall (x, w, \alpha) \in X\times \R_+^N \times \R_+
\end{equation}
we see that finding an equilibrium amounts, for fixed $\alpha$, to solving for $w\in \R_+^N$ the system of $N$ nonlinear equations
\begin{equation}\label{zeroG}
G(w, \alpha)=0, 
\end{equation}
where $G=(G_1, \ldots, G_N)$ is given by 
\begin{equation}\label{defG_i}
G_i(w, \alpha)=\pi'_i(w_i)+ \int_X \frac{\partial R_{\sigma}}{\partial w_i}  (x, w) \wmu(x,w, \alpha) \mbox{d}x.
\end{equation}
We already know that for every $\alpha\geq 0$, \pref{zeroG} admits at least one solution and that all such solutions belong to the compact subset of $(0,+\infty)^N$, $[\uw, \ow]^N$ (see  Lemma \ref{variat}). For $\alpha=0$, $\wmu_0:=\wmu(., w, 0)$ does not depend on $w$ and is the density of the uniform probability measure on $X$, hence for $\alpha=0$ there exists a unique equilibrium which is the unique minimizer of the strictly convex function $J_{\wmu_0}$. Now observing that $G$ is of class $C^1$ on $(0,+\infty)\times \R_+$ and that the Jacobian of $G$ can be written as  a perturbation of order $\alpha$  (see details in the proof below) of a symmetric definite positive matrix, (at least local) uniqueness for small $\alpha$ follows from  the implicit function theorem. More precisely, defining 
\begin{equation}\label{defdealpha0}
\alpha_0:=   \frac{w_0}{N} \min\{ \pi_i''(w), \; i=\{1, \ldots, N\},  \; w \in [\uw, \ow]\},
\end{equation}
 we have:

\begin{thm}\label{thuniq}
In addition to the general assumptions of paragraph \ref{assum}, suppose further that \pref{fisstrict} holds. Defining  $\alpha_0$ as in \pref{defdealpha0}, and 
\[\theta_0:=\frac{\alpha_0}{1+\alpha_0},\]
then, for every $\theta \in [0, \theta_0]$, there exists a unique equilibrium (in the sense of definition \ref{equidef}).

\end{thm}

\begin{proof} Let us denote by $\cdot $ the usual scalar product on $\R^N$ and by $\vert . \vert$ the corresponding euclidean norm on $\R^N$. Take $(w, \alpha) \in [\uw, \ow]^N$, $\alpha\in [0,1]$. 

{\textbf{Step 1:  invertibility of the Jacobian of $G$.}} Let us denote by $A_{ij}:=\frac{\partial G_i}{\partial w_j}(w, \alpha)$ the entries of the Jacobian (with respect to $w$) matrix of $G$ at $(w, \alpha)$:
\begin{equation}\label{jac1}
A_{ij}= \pi_i''(w_i)\delta_{ij} + \int_X \frac{\partial^2 R_{\sigma}}{\partial w_j \partial w_i } \wmu + \int_X \frac{\partial R_{\sigma}}{\partial w_i} \frac{\partial \wmu} { \partial w_j }.  
\end{equation}
In matrix form, this reads as
\[A:= {\mathrm{diag}}( \pi_1''(w_1), \ldots, \pi_N''(w_N))+ B + C\]
where $B$ is symmetric positive definite and $C$ is a mixture of rank one matrices
\[B:= \int_X  D^2_{ww} R_{\sigma}(x,w) \wmu (x, w, \alpha) \mbox{d}x, \; C:= \int_X  \nabla_w R_{\sigma}(x,w) \nabla_w \wmu  ^{\top} (x, w, \alpha) \mbox{d}x. \]
Hence, by Cauchy-Schwarz inequality,  for $\xi \in \R^N$, we  have
\begin{equation}\label{jac2}
 A \xi \cdot \xi \geq   \Big( \nu -  \int_X \vert \nabla_w R_\sigma(x, w)\vert \vert \nabla_w \wmu(x, w, \alpha) \vert \mbox{d} x  \Big) \vert \xi\vert^2 + B\xi \cdot \xi 
 \end{equation}
for 
 \begin{equation}\label{defdenu}
 \nu:=\min_{i=1,  \ldots, N} \pi_i''(w_i) \geq \min\{ \pi_i''(w), \; i=\{1, \ldots, N\},  \; w \in [\uw, \ow]\}.
\end{equation}
Since, by \pref{probaix}
\begin{equation}\label{petitesdp}
0 \leq \frac{\partial R_{\sigma}}{\partial w_i} (x,w) \leq 1
\end{equation}
we have 
\begin{equation}\label{petitgradR}
 \vert \nabla_w R_\sigma(x, w)\vert \leq \sqrt{N}.
\end{equation}
Let us now estimate the partial derivatives of $\wmu$ with respect to the $w_j$'s:
\[\frac{\partial \wmu}{\partial w_j}(x, w, \alpha) = \frac{\alpha R_{\sigma}^{\alpha-1}(x,w) }{\int_X R_{\sigma}^\alpha}  \frac{\partial R_{\sigma}}{\partial w_j}(x,w) -\alpha \frac{R_{\sigma}^{\alpha}(x,w)}{\Big(\int_X R_{\sigma}^\alpha\Big)^2}\int_X R_{\sigma}^{\alpha-1}    \frac{\partial R_{\sigma}}{\partial w_j}   \]
together with \pref{petitesdp} and $R_\sigma \geq w_0$ so that $R_\sigma^{\alpha-1} \leq \frac{R_\sigma^{\alpha}}{w_0}$, this yields
\[0 \leq  \frac{R_{\sigma}^{\alpha-1}(x,w) }{\int_X R_{\sigma}^\alpha}  \frac{\partial R_{\sigma}}{\partial w_j}(x,w) \leq \frac{ R_{\sigma}^{\alpha}(x,w) }{w_0 \int_X R_{\sigma}^{\alpha} }    \]
and
\[0 \leq  \frac{R_{\sigma}^{\alpha}(x,w)}{\Big(\int_X R_{\sigma}^\alpha\Big)^2}\int_X R_{\sigma}^{\alpha-1}    \frac{\partial R_{\sigma}}{\partial w_j} \leq  \frac{ R_{\sigma}^{\alpha}(x,w) }{w_0 \int_X R_{\sigma}^{\alpha} } \]
so that
\[\Big \vert \frac{\partial \wmu}{\partial w_j}(x, w, \alpha)\Big \vert \leq \frac{  \alpha}{w_0} \frac{R_{\sigma}^{\alpha}(x,w)}{\int_X R_{\sigma}^\alpha}\]
hence 
\begin{equation}\label{petitgradwmu}
\int_X   \vert \nabla_w \wmu(x, w, \alpha) \vert \mbox{d} x \leq \frac{ \alpha \sqrt{N} }{w_0}.
\end{equation}
Using \pref{petitgradR}-\pref{petitgradwmu} in \pref{jac2}, we thus get
\[A \xi \cdot \xi \geq \Big(\nu- \frac{ \alpha N}{w_0}\Big) \vert \xi\vert^2 + B\xi \cdot \xi\]
which, since $B$ is symmetric positive definite,   enables us to conclude with \pref{defdenu} that $A$ is invertible whenever $\alpha \in [0, \alpha_0]$.

\smallskip

{\textbf{Step 2:  uniqueness.}} Let us define
\[I:=\{\alpha \in [0, \alpha_0] \; : \mbox{ there exists a unique } w \in [\uw, \ow]^N \mbox{ such that } G(w, \alpha)=0\}.\]
We have already observed that $0\in I$ so that $I$ is nonempty, so proving that $I$ is closed and open in $[0, \alpha_0]$ will yield the desired uniqueness result. We can deduce from the previous step and the implicit function theorem that for every $\alpha \in [0, \alpha_0]$ and every $w\in \R_+^N$ such that $G(w, \alpha)=0$ there is some neighbourhood of $(\alpha, w)$ say $((\alpha-\eps, \alpha+\eps)\cap [0, \alpha_0])\times B(w, r)$ for some $\eps>0$, $r>0$ and a $C^1$ curve $\gamma_{\alpha, w}$: $(\alpha-\eps, \alpha+\eps) \cap [0, \alpha_0] \to \R^N$ such that  $G^{-1}(\{0\}) \cap( B(w, r) \times ((\alpha-\eps, \alpha+\eps) \cap [0, \alpha_0]))=\{(\gamma_{\alpha, w}(\alpha'), \alpha' ), \; \alpha'\in (\alpha-\eps, \alpha+\eps) \cap [0, \alpha_0]\}$. Let $(\alpha_n)_n \in I^{\N}$ converge to some $\alpha$; if $\alpha$ was not in $I$, we could find $w\neq \hat{w}$ with $G(w, \alpha)=G(\hat{w}, \alpha)=0$, but then for $n$, large enough we would have $\gamma_{\alpha, w} (\alpha_n) \neq \gamma_{\alpha, \hat{w}} (\alpha_n)$ contradicting the fact that $\alpha_n\in I$. Hence $I$ is closed. Now let $\alpha \in I$ and $w$ be the only root of $G(w, \alpha)=0$. If $I$  was not a neighbourhood of $\alpha$ in $[0, \alpha_0]$, we could find a sequence $(\alpha_n)$ in $[0, \alpha_0]\setminus I$ converging to $\alpha$. For each $n$, we could pick $w_n\neq \hat{w}_n$ with $G(\hat{w}_n, \alpha_n)=G(w_n, \alpha_n)$, since both sequences $(w_n)$ and $(\hat{w}_n)_n$ take values in the compact set $[\uw, \ow]^N$, both $(w_n)$ and $(\hat{w}_n)$ should converge to $w$ so that, for large enough $n$,  we should have $w_n=\hat{w}_n=\gamma_{\alpha, w}(\alpha_n)$ which yields the desired contradiction. We thus have shown that $I$ is open in $[0, \alpha_0]$.

\end{proof}

\subsection{The  un-regularized case (zero noise limit)}\label{nonoise}

We now consider the case where $\sigma=0$, for which the revenue of workers is given by \pref{defr0}. In this case, given $w\in \R_+^N$, the set of workers locations for which working for firm $i$ (or staying home for wage $w_0$ if $i=0$) is optimal is
\begin{equation}\label{defvi}
V_i(w):=\{x \in X \; : \; R_0(x,w)= w_i-c_i(x)\}, \; i=0, \ldots, N.
\end{equation}
In case there are ties, i.e. several optimal choices, we also define the set of workers locations for which $i$ is strictly prefered to the other options:
\begin{equation}\label{defvis}
V_i^s(w):= V_i(w) \setminus \cup_{j \neq i} V_j(w).
\end{equation}
Equilibrium on the residential housing market gives the population density as a function of $R_0$:
\begin{equation}\label{defmuw0}
 \mu_w(x):= \frac{ R_0(x, w)^{\frac{\theta}{1-\theta}}   } {\int_X R_0(y, w)^{\frac{\theta}{1-\theta}} \mbox{d} y   }, \; \forall x\in X.
\end{equation}
The labour demand of firm $i$ is determined by \pref{Lienv} exactly as in paragraph \ref{assum}. As for the total labour supply for location $i=1, \ldots, N$, because of possible ties\footnote{An easy way to rule out ties is the following. Since $\mu_w$ is absolutely continuous, one way to ensure that $V_i(w)\setminus V_i^s(w)$ is neglible is to assume that for every $i,j$ with $i\neq j$ and any $\lambda\in \R$, the level set $\{x\in X \; :  \; c_i(x)-c_j(x)=\lambda\}$ is Lebesgue-negligible.}, it has to lie in the interval $[\mu_w(V_i^s(w)), \mu_w(V_i(w))]$. Equilibrium on the labour market then reads 
\begin{equation}\label{lme01}
-\pi'_i(w_i) \in [\mu_w(V_i^s(w)), \mu_w(V_i(w))], \; \forall i=\{1, \ldots, N\}
\end{equation}
supplemented with an additional consistency condition: workers who are not hired by any firm are those for which  $w_0$ is optimal:
\begin{equation}\label{lme00}
1+\sum_{i=1}^N \pi'_i(w_i) \in [\mu_w(V_0^s(w)), \mu_w(V_0(w))].
\end{equation}
In this context, a spatial equilibrium is a collection of positive wages $w$ for which \pref{lme01}-\pref{lme00} are satisfied with $\mu_w$ given by \pref{defmuw0}. Existence is ensured by:

\begin{prop}
There exists $w\in (0, +\infty)^N$ such that  \pref{lme01}-\pref{lme00} are satisfied with $\mu_w$ given by \pref{defmuw0}.
\end{prop}

\begin{proof}
Eventhough a fixed-point proof is possible, we prefer to give a short proof by passing to the (zero-noise)  limit in the regularized equilibria whose existence is ensured by theorem \ref{existeqsig}. Indeed, we have seen (lemma \ref{variat} and theorem \ref{existeqsig})  that for every $\sigma\in (0,1)$, there exists $w^\sigma$ belonging to the (independent of $\sigma \in (0,1)$) compact subset $[\uw, \ow]^N$ of $(0, + \infty)^N$ such that
\begin{equation}\label{eqapprox}
-\pi'_i(w_i^\sigma)=  \int_X \frac{e^\frac{w_i^\sigma-c_i(x)}{\sigma}}{\sum_{j=0}^{N} e^\frac{w_j^\sigma-c_j(x)}{\sigma}}  \mu_\sigma(x) \mbox{d}x, \; \forall i\in \{1, \ldots, N\}
\end{equation}
where
\begin{equation}
\mu_\sigma(x)= \frac{ R_\sigma(x, w^\sigma)^{\frac{\theta}{1-\theta}}   } {\int_X R_\sigma(y, w^\sigma)^{\frac{\theta}{1-\theta}} \mbox{d} y   }, \; \forall x\in X.
\end{equation}
Up to a subsequence, we may assume that $w^\sigma$ converges to some $w\in [\uw, \ow]^N$ as $\sigma \to 0^+$. This implies that $\pi'_i(w_i^\sigma)$ converges to $\pi'_i(w_i)$ and   $(R_\sigma(., w^{\sigma}), \mu_{\sigma})$ converges uniformly to $(R_0(., w), \mu_w)$. Moreover, for $i=0, \ldots, N$,
\[x\in V^{i}_s(w) \Rightarrow  \lim_{\sigma \to 0^+} \frac{e^\frac{w_i^\sigma-c_i(x)}{\sigma}}{\sum_{j=0}^{N} e^\frac{w_j^\sigma-c_j(x)}{\sigma}} =1\]
and
\[x \in X\setminus V_i(w)  \Rightarrow \lim_{\sigma \to 0^+} \frac{e^\frac{w_i^\sigma-c_i(x)}{\sigma}}{\sum_{j=0}^{N} e^\frac{w_j^\sigma-c_j(x)}{\sigma}} =0.\]
So letting $\sigma \to 0^+$ in \pref{eqapprox}, we easily get that $w$ satisfies \pref{lme01}. As for \pref{lme00}, we remark that \pref{eqapprox} implies
\[1+\sum_{i=1}^N \pi'_i(w_i^\sigma)= \int_X \frac{e^\frac{w_0}{\sigma}}{\sum_{j=0}^{N} e^\frac{w_j^\sigma-c_j(x)}{\sigma}}  \mu_\sigma(x) \mbox{d}x\]
by letting $\sigma\to 0^+$, we obtain \pref{lme00}. 

\end{proof}

\section{A teleworking model}\label{sec-tele}

We now consider a variant of the previous model where teleworking is introduced. Remote (teleworking) workers have no commuting costs and the production function of each firm $i$ depends on both the numbers of on-site workers which we denote $l_i^1$ and the numbers of teleworkers $l_i^2$. We denote by $\tf_i$: $\R_+^2 \to \R_+$ this production function and assume that for $i=1, \ldots, N$, $\tf_i$ is increasing in both arguments, strictly concave and continuous on $\R_+^2$ and $C^1$ on $(0, +\infty)^2$ as well as
\begin{equation}\label{hyptfi1}
\tf_i(0, 0)=0, 
\end{equation}
\begin{equation}\label{hyptfi2}
\lim_{l\to +\infty} \tf_i(l, s) = \lim_{l\to +\infty} \tf_i(s, l)=+\infty, \forall s>0. 
\end{equation}
and
\begin{equation}\label{hyptfi3}
\lim_{(l,s) \in \R_+^2, \; l+ s \to +\infty} \frac{\tf_i(l, s)}{l+s} = 0, 
\end{equation}
which is the case as soon as $\tf_i$ is homonogeneous of degree strictly less than $1$. Typical examples of production functions which fulfill the previous conditions are Cobb-Douglas functions
\[\tf_i(l,s)= A_i l^{\alpha_i} s^{\beta_i}, \; \alpha_i>0, \; \beta_i>0, \; \alpha_i +\beta_i <1\]
or constant elasticity of substitution functions
\[\tf_i(l,s):= A_i (l^{\alpha_i}+ B_i s^{\alpha_i})^{\frac{\beta_i}{\alpha_i}}, \; A_i>0, \; B_i>0, (\alpha_i, \beta_i) \in (0,1)^2.\]

\smallskip

Let us define, for every $\tw=(\tw^1, \tw^2)\in \R_+^2$, the profit of firm $i$ when the vector of wages for on site/remote work is $\tilw$:
\begin{equation}
\tpi_i(\tw)=\sup_{l^1 \geq 0, \; l^2 \geq 0} \{\tf_i(l^1, l^2)-w^1 l^1-w^2 l^2\}.
\end{equation}
It follows from \pref{hyptfi1}, \pref{hyptfi2}, \pref{hyptfi3} that $\tpi_i$ is strictly convex and nonincreasing in both arguments on $\R_+^2$, with
\begin{equation}\label{tpcoerc0}
\tpi_i(0,w)=\tpi_i(w,0)=\lim_{\eps \to 0^+} \tpi_i(\eps, w)=\lim_{\eps \to 0^+} \tpi_i(w, \eps)= +\infty, \; \forall w\geq 0
\end{equation}
and $\tpi_i$ is $C^1$ on $(0,+\infty)$ and for every $(w^1, w^2)\in (0,+\infty)$ and for $k=1, 2$:
\begin{equation}\label{enveltpi}
\frac{\partial \tpi_i}{\partial w^k} (w^1, w^2)=-L_i^{k}(w^1, w^2)
\end{equation}
where 
\begin{equation}\label{demtpi}
(L_i^1(w^1, w^2), L_i^2(w^1, w^2))=\argmax_{l^1\geq 0, \; l^2\geq0 } \{\tf_i(l^1,l^2)-w^1 l^1-w^2 l^2\}
\end{equation}
so that $L_i^1$ represents the demand of on-site labour and $L_i^2$ the demand for remote labour of firm $i$. 

\smallskip

Given a collection of wages $\tw:=(\tw_1, \ldots, \tw_N)\in \R_+^{2N}$, where $\tw_i:=(w_i^1, w_i^2)$  (the superscript $k=1$ corresponding to on site work and $k=2$ to remote work), the revenue of workers living at $x$ takes the form
\begin{equation}
\tR_\sigma(x, \tw)=\sigma  \log \Big( e^{\frac{w_0}{\sigma}} + \sum_{i=1}^N  e^{\frac{w_i^1-c_i(x)}{\sigma}}+  \sum_{i=1}^N  e^{\frac{w_i^2}{\sigma}}\Big).
\end{equation}
Arguing as in section \ref{assum}, assuming the Cobb-Douglas form \pref{cobbdu} for workers' preferences, equilibrium on the residential market implies that the density of workers can be expressed as:
\begin{equation}\label{tilmuw}
\tmu_{\tw}(x):=\frac{ \tR_\sigma(x, \tw)^{\frac{\theta}{1-\theta}}   } {\int_X \tR_\sigma(y, \tw)^{\frac{\theta}{1-\theta}} \mbox{d} y   }, \; \forall x\in X.
\end{equation}
Again arguing as in  in section \ref{assum}, the supply of on-site labour for firm $i$ is given by
\begin{equation}\label{laboursupplyonsite}
\int_X \frac{\partial \tR_\sigma}{\partial w_i^1}(x,w) \tmu_{\tw}(x) \mbox{d}x=  \int_X \frac{e^\frac{w_i^1-c_i(x)}{\sigma}}{  e^{\frac{w_0}{\sigma}} + \sum_{j=1}^N  e^{\frac{w_j^1-c_j(x)}{\sigma}}+  \sum_{j=1}^N  e^{\frac{w_j^2}{\sigma}  }}  \tmu_{\tw} (x) \mbox{d}x. 
\end{equation}
and the supply of remote labour is given 
\begin{equation}\label{laboursupplyremote}
\int_X \frac{\partial \tR_\sigma}{\partial w_i^2}(x,w) \tmu_{\tw}(x) \mbox{d}x=  \int_X \frac{e^\frac{w_i^2}{\sigma}}{  e^{\frac{w_0}{\sigma}} + \sum_{j=1}^N  e^{\frac{w_j^1-c_j(x)}{\sigma}}+  \sum_{j=1}^N  e^{\frac{w_j^2}{\sigma}  }}  \tmu_{\tw} (x) \mbox{d}x. 
\end{equation}
In this setting, an equilibrium is a collection of wages $\tw:=(\tw_1, \ldots, \tw_N)=(w_1^1, w_1^2, \ldots, w_N^1, w_N^2)\in \R_+^{2N}$ for which supply and demand for on-site and remote labour coincide i.e. 
\begin{equation}\label{equitele}
0=\frac{\partial \tpi_i}{\partial w^k} (w_i^1, w_i^2)+\int_X \frac{\partial \tR_\sigma}{\partial w_i^k}(x,w) \tmu_{\tw}(x) \mbox{d}x, \; i=1, \ldots, N, k=1,2.
\end{equation}
Defining for $\mu\in \PP(X)$, the strictly convex function, 
\begin{equation}\label{deftj}
\tJ_\mu(\tw):=\sum_{i=1}^N \tpi_i(w_i^1, w_i^2)+ \int_X  \tR_\sigma (x,\tw) \mbox{d} \mu(x), \; \forall \tw=(w_1^1, w_1^2, \ldots, w_N^1, w_N^2)\in \R_+^{2N}
\end{equation}
the system of $2N$ equations for equilibrium wages \pref{equitele} reads
\begin{equation}\label{fpeqrem}
\tw =\widetilde{\Phi}(\tw):=\argmin \tJ_{\tmu_{\tw}}.
\end{equation}
Under the assumptions of this paragraph, the existence and uniqueness results from section \ref{sec-exist} can be extended to the case of teleworking (details can be found in Appendix B):
\begin{itemize}
\item there exist equilibria i.e. solutions to \pref{equitele},

\item if in addition, the production functions $\tf_i$ are of class $C^2$ on $(0,+\infty)^2$ and 
\begin{equation}\label{strongconcaver}
D^2 \tf_i(l^1, l^2) \mbox{ is  negative definite, for every $(l^1,l^2)\in (0,+\infty)^2$},
\end{equation}
then \pref{equitele} admits a unique solution provided $\theta$ is small enough (an explicit bound can be found in Appendix B).
\end{itemize}

\section{Numerical simulations}\label{sec-num}

We approximate an equilibrium by solving equation \eqref{eqlabour} with the distribution $\mu=\mu_w$ given by \eqref{defmuw}, i.e. we are solving
\begin{equation}\label{sec:num:eqlabour}
\pi'_i(w_i)+ \int_X \frac{\partial R_\sigma}{\partial w_i}(x,w) \mu_w(x) \mbox{d}x=0, \; \forall i\in \{1, \ldots, N\}.
\end{equation}
This amounts to computing the zero of a function for which we apply a modification of the Powell hybrid method, as detailed below. 

\subsection{The scheme}
For convenience, we focus on the case when $X = [0,1]$ but the secheme can directly be extended to the case when $X$ is a bounded domain of $\R^d$, with $d>1$.  Let $X_h$ be a uniform grid on $X$ with step $h = 1/N_h$, $N_h\in\N$. Let $x_k$ denote a generic point in $X_h$; the values of $R_\sigma(\cdot,w)$, $\frac{\partial R_\sigma}{\partial w^i}(\cdot,w)$ and $\mu_w$ at $x_k$ will be respectively denoted by $R_k$, $DR_k^i$ and $\mu_k$ and computed by the explicit form of  $R_\sigma(\cdot,w)$, $D_wR_\sigma(\cdot,w)$ and $\mu_w$. We use the trapezoidal rule to estimate the value of the integrals. Therefore, we make the following approximation:
\begin{displaymath}
\int_X \frac{\partial R_\sigma}{\partial w_i}(x,w) \mu_w(x)\simeq h\left(\frac{DR_0\mu_0+DR_{N_h}\mu_{N_h}}{2}+\sum_{k=1}^{N_h-1}DR_k\mu_k \right).
\end{displaymath}
Note that when the dimension of $X$ is greater than one, the scheme can be adapted by considering another approximation of the integrals.

\subsection{The method}
Solving \eqref{sec:num:eqlabour} consists in finding a root of a vector function $F:\R^d\rightarrow \R^d$. We use a modification of the Powell hybrid method \cite{Powell1970, More1980} which consists in finding a minimizer of the function
\begin{displaymath}
	G= \sum_{i=1}^N F_i^2
\end{displaymath}
by constructing a sequence. For each iteration $n$, the next element $x_{n+1} = x_{n} + \delta_n$ is computed with the step $\delta_n$, which is a convex combination of the Gauss-Newton and the scaled gradient descent step. The Jacobian matrix of $F$ is computed with a forward-difference approximation for the initial point, and then updated with the Broyden's iteration, see \cite{Coleman1984} for more details. As explained in \cite{More1980}, this method "guarantees (under reasonable conditions) global convergence and a fast rate of convergence".


\subsection{Numerical results: some comparative statics}

We present two simulations whose goal is to analyze the influence of some  parameters on the equilibrium. They are carried out on a one-dimensional domain while  a third simulation on a two-dimensional domain will be presented in paragraph \ref{sim2d}. \\
The setting is as follows: we set $X=[-10,10]$, and we assume that there are three workplaces ($N = 3$) located at three different points in $X$. 
Let $y_i\in X $ be the location of the $i$th workplace. We assume that
\begin{displaymath}
	y_1 = -7,\quad y_2 = 0, \quad y_3 = 3,
 \end{displaymath}
 and that each workplace corresponds to a firm that seeks to maximize its profits; with a slight abuse of language,
 let $y_i$ be the name of the firm located at $y_i$. We assume that the transport cost to reach $y_i$ from $x\in X$ is given by
\begin{displaymath}
	c_i(x) = \frac{\abs{x-y_i}}{2},\quad\quad \forall i\in\{1,2,3\}.
\end{displaymath}
Note that we could have used any other continuous function on $X$ to model the transport costs without changing the scheme.

\subsubsection{Comparative statics as the preference parameter $\theta$ varies}
\paragraph{Definition of the model and the parameters}
We assume that the production of the firm $y_i$ is given by
\begin{displaymath}
	f_i(l) = A^{1-\beta} l^{\beta},\quad \forall l\in[0,+\infty).
 \end{displaymath}
The parameter $A$ can naturally be interpreted as the firm's productivity (which may depend on its capital for instance). The parameters used in Test 1 are listed in Table \ref{table:chap:spatial_numerics:table:test1} below.

\begin{center}
\begin{tabular}{|c|c|}
\hline Parameter & Value \\ \hline
$\beta$ & 0.7 \\ \hline
$A $ & $ 10^4$ \\ \hline
$w_0 $ & 12 \\ \hline
$\sigma $ & $0.1$ \\ \hline
\end{tabular}
\captionof{table}{\label{table:chap:spatial_numerics:table:test1}The parameters used in Test 1}
\end{center}

\paragraph{Numerical results}
In the following three figures, we compare the results obtained for different values of $\theta$.
In Figure \ref{fig:chap:spatial_numerics:sec:theta:distributions_test1}, we display the residential distribution of the people working at the different workplaces. Recall that these distributions are given by 
\begin{displaymath}
	X\ni x\mapsto \frac{\partial R_\sigma}{\partial w_i}(x,w)\mu(x),\qquad\forall i\in\{0,1,2,3\}.
\end{displaymath}
The lines 
(\tikz[baseline=-\the\dimexpr\fontdimen22\textfont2\relax,inner sep=0pt] \draw[orange2,line width=1pt](0,0) -- (5mm,0);), 
(\tikz[baseline=-\the\dimexpr\fontdimen22\textfont2\relax,inner sep=0pt] \draw[bleu2,line width=1pt](0,0) -- (5mm,0);) and
(\tikz[baseline=-\the\dimexpr\fontdimen22\textfont2\relax,inner sep=0pt] \draw[vert2,line width=1pt](0,0) -- (5mm,0);) 
are associated to the residences of the agents working at $y_1$, $y_2$ and $y_3$ respectively. The curve 
(\tikz[baseline=-\the\dimexpr\fontdimen22\textfont2\relax,inner sep=0pt] \draw[gris2,line width=1pt](0,0) -- (5mm,0);) 
corresponds to the residential distribution of \emph{independent workers} (i.e. those who stay home for the wage $w_0$). 
\\
In Figure \ref{fig:chap:spatial_numerics:sec:theta:salaries_masses_test1}, we display the wages and the number of workers in each workplace (we use the same color code as in Figure \ref{fig:chap:spatial_numerics:sec:theta:distributions_test1}).
Finally, the rental price as a function of the spatial variable $x$ is plotted in Figure \ref{fig:chap:spatial_numerics:sec:theta:rental_price_test1}.

\begin{center}
	\includegraphics{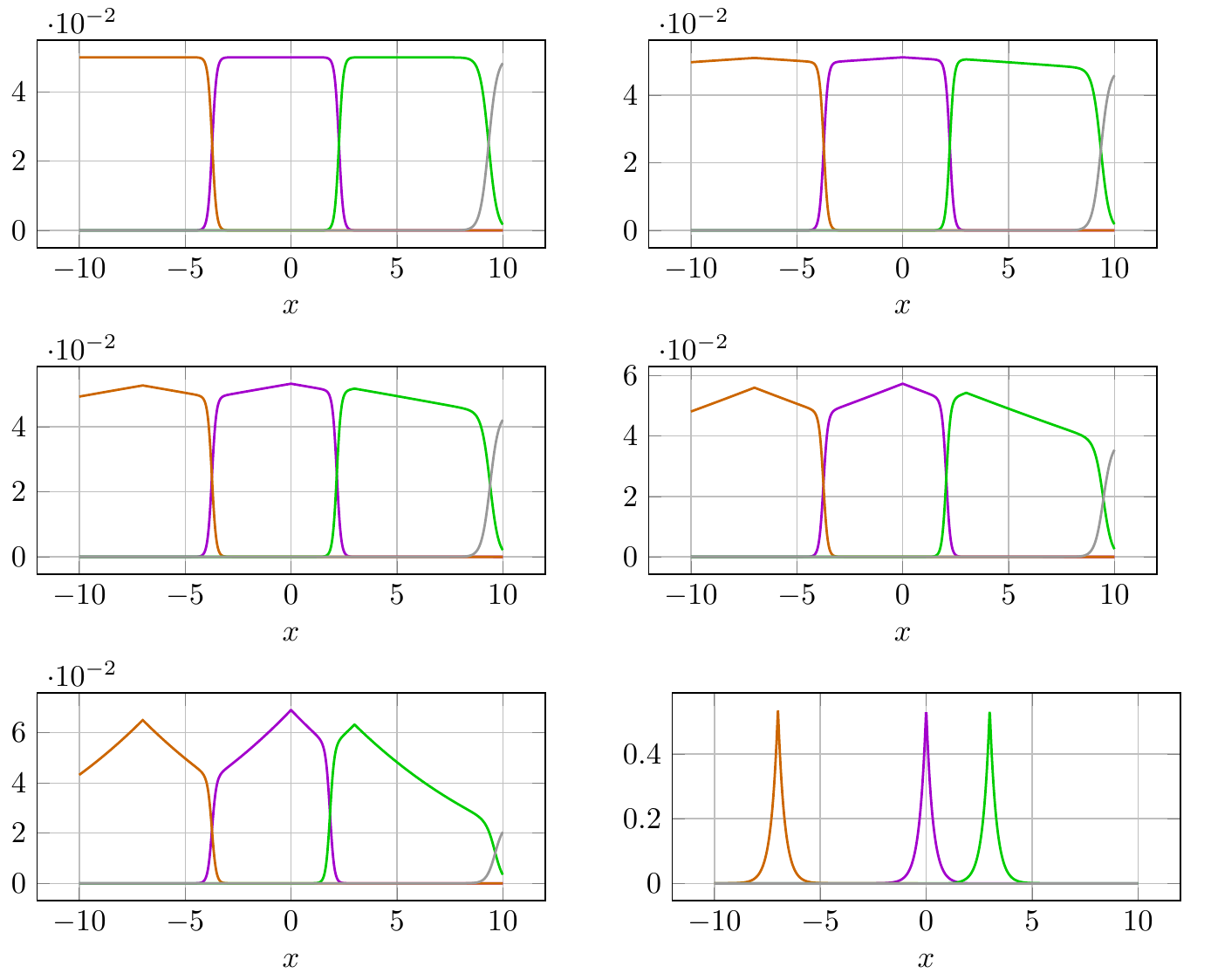}
\captionof{figure}{\label{fig:chap:spatial_numerics:sec:theta:distributions_test1}Residential distributions  of the people working at the different workplaces and of the independent workers, with $\theta = 0$ (top-left), $\theta = 0.2$ (top-right), $\theta = 0.4$ (middle-left), $\theta = 0.6$ (middle-right), $\theta = 0.8$ (bottom-left), $\theta = 0.99$ (bottom-right).}
\end{center}

Let us focus on the case when $\theta = 0$. The agents have only one source of utility, the surface that they rent.
Therefore, as it clearly appears on Figure \ref{fig:chap:spatial_numerics:sec:theta:distributions_test1},
they distribute themselves uniformly on $X$ (the supply of space is constant). This gives to $y_3$ a positional advantage; indeed
the basin of attraction of $y_3$ is larger than those of $y_1$ and $y_2$.
Therefore, the supply of labour at $y_3$ is larger than at $y_1$ or $y_2$. We see that different locations lead to differences in labour supply.\newline 
Due to the advantage that $y_3$ has, it attracts more workers and may pay them less than $y_1$ or $y_2$. Similarly, $y_1$ has a positional advantage with respect to $y_2$. This explains why $w_3<w_1<w_2$, even though $y_3$ attracts more workers than $y_1$, which attracts more workers than $y_2$, see Figure \ref{fig:chap:spatial_numerics:sec:theta:salaries_masses_test1}.

\begin{center}
\includegraphics{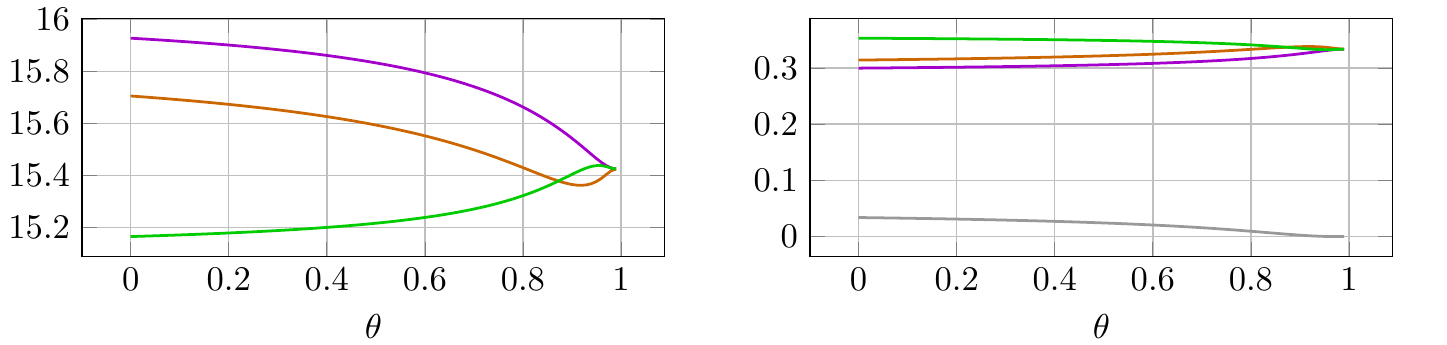}
\captionof{figure}{\label{fig:chap:spatial_numerics:sec:theta:salaries_masses_test1}Wages versus $\theta$ (on the left) and the number of workers versus $\theta$ (on the right).}
\end{center}

\begin{center}
\includegraphics{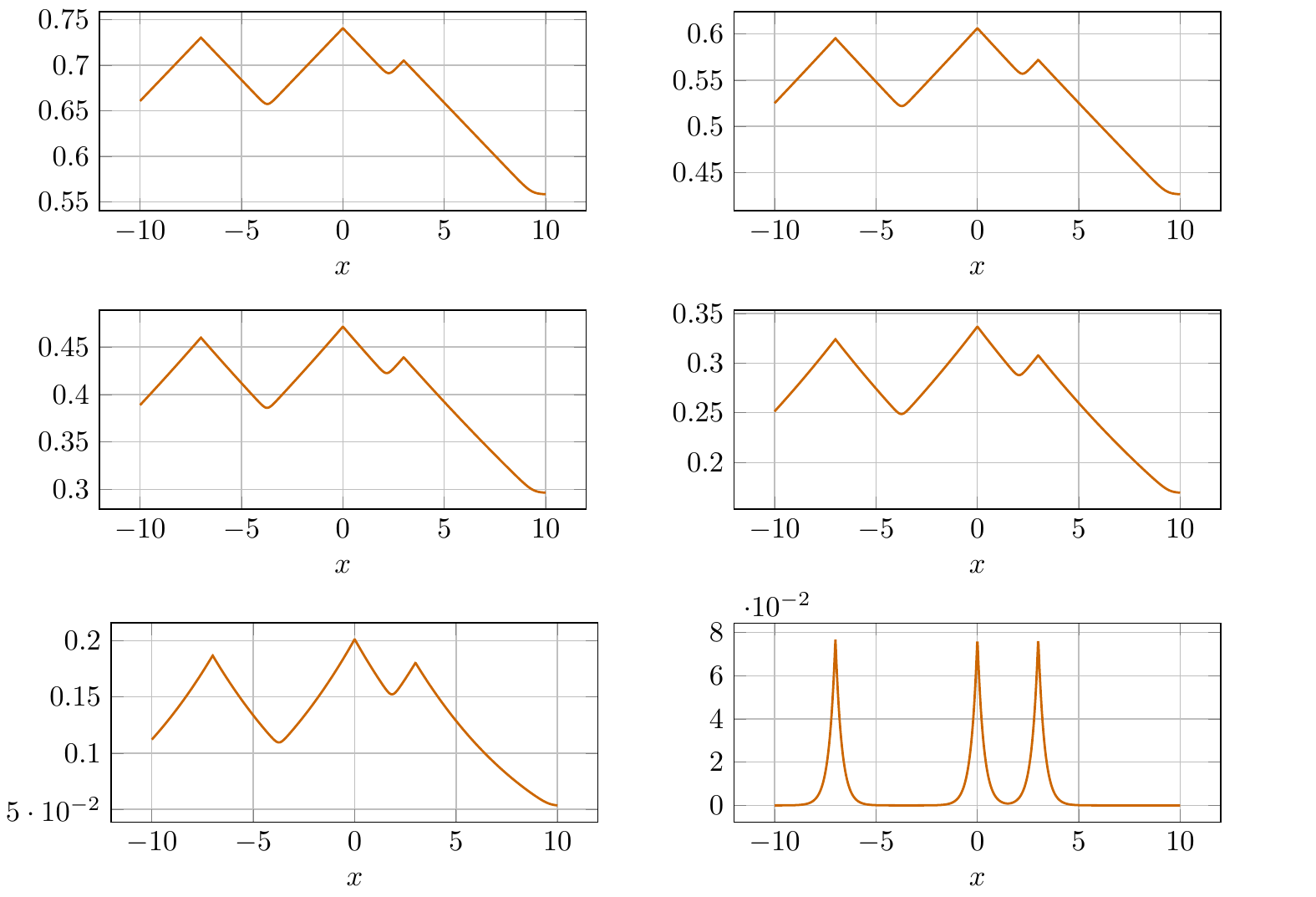}
\captionof{figure}{\label{fig:chap:spatial_numerics:sec:theta:rental_price_test1}Rental price versus $x$: comparison for $\theta = 0$ (top-left), $\theta = 0.2$ (top-right), $\theta = 0.4$ (middle-left), $\theta = 0.6$ (middle-right), $\theta = 0.8$ (bottom-left), $\theta = 1$ (bottom-right).}
\end{center}

As $\theta$ increases, the relative importance of the surface of the house in the utility function of the agents decreases.
As a consequence, the demand for surface decreases and so does the rental price, see Figure \ref{fig:chap:spatial_numerics:sec:theta:rental_price_test1}. On the other hand, the relative importance of consumption in the utility function of the agents increases. Therefore, the tendency is to have a concentration of housing close to the workplaces, because the agents choose to reduce their transport costs in order to increase their consumption, see Figure \ref{fig:chap:spatial_numerics:sec:theta:distributions_test1}. In Figure \ref{fig:chap:spatial_numerics:sec:theta:salaries_masses_test1}, we observe that, when $\theta$ varies from $0$ to $0.8$,  wages at $y_1$ and $y_2$ decrease, while  wages at $y_3$ increase. We may give two reasons for that. First, the concentration of houses near the workplaces tends to reduce the competition on the labour market between $y_1$ and $y_2$. Second, the number of workers living in the interval $[5,10]$ decreases, so that the positional advantage of $y_3$ decreases. 

For $\theta> 0.8$, the concentration phenomenon progressively isolates $y_1$ from $y_2$ and $y_3$. Therefore, $y_1$ enjoys a positional advantage similar to the one that $y_3$ had when $\theta$ was small. This allows $y_1$ to decrease the level of wage. On the other hand, this extra competition pushes $y_3$ to increase its wage.

Finally, when $\theta$ is close to $1$, the size of the basins of attraction of the different workplaces becomes small, so that they are isolated of each other. 
When $\theta=0.99$, the wages and the number of agents in each workplace is almost the same.

\subsubsection{Comparative statics as the productivity of teleworkers varies}
\paragraph{Description of the model and the parameters}
We use the teleworking model described in section \ref{sec-tele}. Therefore, we approximate a solution of \eqref{equitele}, the analogue equation of \eqref{sec:num:eqlabour}, by using the same method introduced in the beginning of this section.
We assume that the production of $y_i$ is given by 
\begin{displaymath}
	 \tf_i(l,s) = A^{1-\beta} (l^{\alpha} + Bs^{\alpha})^\frac{\beta}{\alpha} ,\quad \forall (l,s)\in[0,+\infty)^2.
 \end{displaymath}
The parameter $B$ is related to the productivity of the teleworkers. We are going to let $B$ vary from 0 to 1.

The parameters used in Test 2 are listed in Table \ref{table:chap:spatial_numerics:table:test3} below.

\begin{center}
\begin{tabular}{|c|c|}
\hline Parameter & Value \\ \hline
$\beta$ & 0.7 \\ \hline
$\alpha$ & 0.9 \\ \hline
$A $ & $10^4$ \\ \hline
$w_0 $ & 12 \\ \hline
$\sigma $ & $0.1$ \\ \hline
$\theta $ & $0.7$ \\ \hline
\end{tabular}
\captionof{table}{\label{table:chap:spatial_numerics:table:test3}The parameters used in Test 2.}
\end{center}

\paragraph{Numerical results}

In the three figures below, we compare the results obtained for different values of $B$.
In Figure \ref{fig:chap:spatial_numerics:sec:telecommuting:distribution},
we display the residential distributions for the workers of  the different workplaces.
The lines 
(\tikz[baseline=-\the\dimexpr\fontdimen22\textfont2\relax,inner sep=0pt] \draw[orange3,line width=1pt](0,0) -- (5mm,0);), 
(\tikz[baseline=-\the\dimexpr\fontdimen22\textfont2\relax,inner sep=0pt] \draw[bleu3,line width=1pt](0,0) -- (5mm,0);) and
(\tikz[baseline=-\the\dimexpr\fontdimen22\textfont2\relax,inner sep=0pt] \draw[vert3,line width=1pt](0,0) -- (5mm,0);) are associated to the residences of the commuters working at $y_1$, $y_2$ and $y_3$ respectively. The curves 
(\tikz[baseline=-\the\dimexpr\fontdimen22\textfont2\relax,inner sep=0pt] \draw[orange1,line width=1pt, densely dashed](0,0) -- (5mm,0);), 
 (\tikz[baseline=-\the\dimexpr\fontdimen22\textfont2\relax,inner sep=0pt] \draw[bleu1,line width=1pt, densely dashed](0,0) -- (5mm,0);) and 
(\tikz[baseline=-\the\dimexpr\fontdimen22\textfont2\relax,inner sep=0pt] \draw[vert1,line width=1pt, densely dashed](0,0) -- (5mm,0);) 
are respectively associated to the residences of the teleworkers working for $y_1$, $y_2$ and $y_3$. The curve 
(\tikz[baseline=-\the\dimexpr\fontdimen22\textfont2\relax,inner sep=0pt] \draw[gris2,line width=1pt](0,0) -- (5mm,0);) 
corresponds to the residences distribution of independent workers.
In Figure \ref{fig:chap:spatial_numerics:sec:telecommuting:salaries_masses}, we plot the wages and the number of workers in each workplace versus $B$ (we use the same color code). In Figure \ref{fig:chap:spatial_numerics:sec:telecommuting:rental_price}, we plot the rental price versus $x$.

\begin{center}
\includegraphics[scale=1]{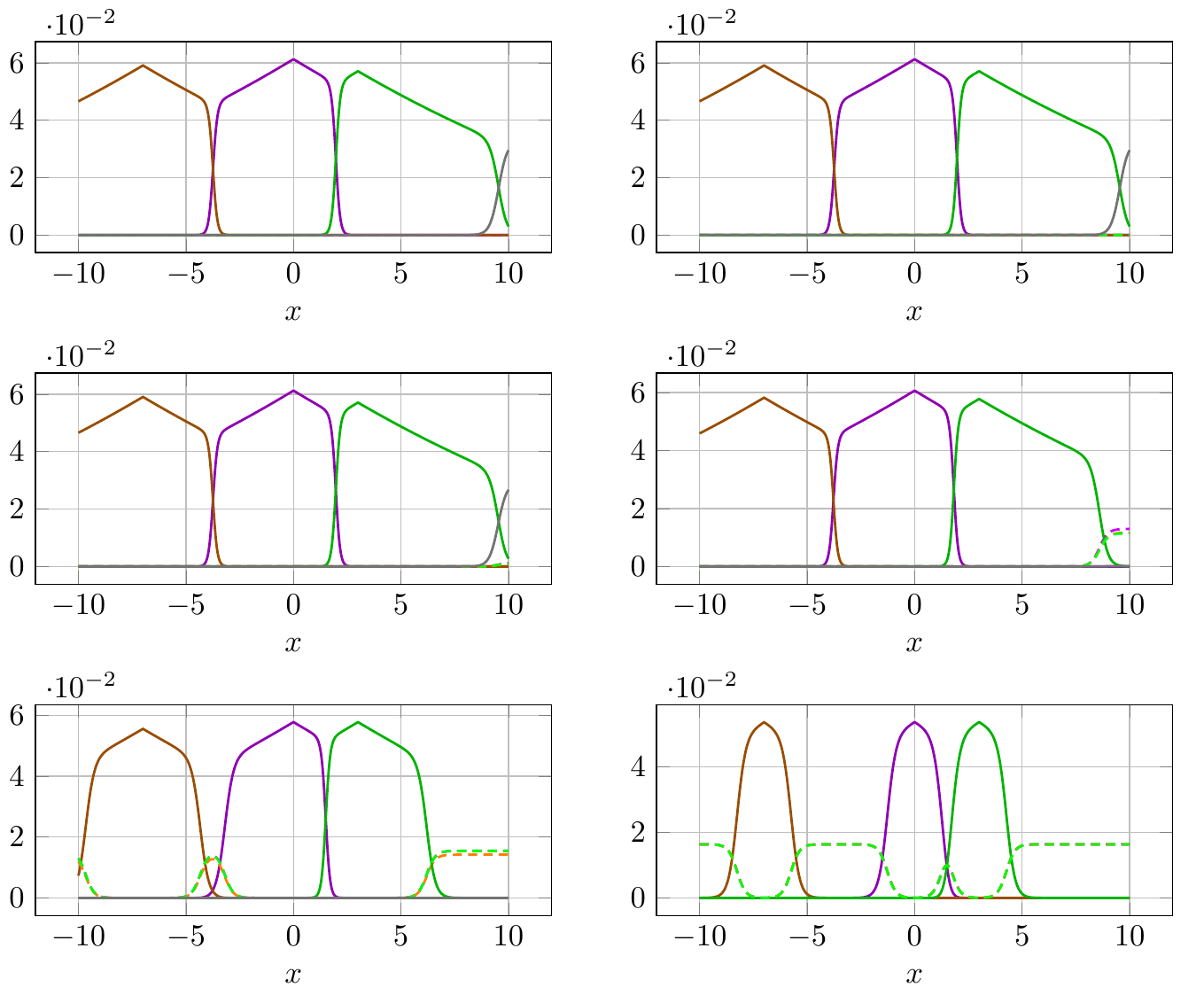}
\captionof{figure}{\label{fig:chap:spatial_numerics:sec:telecommuting:distribution}Residential distributions of the people working for the different workplaces and of the independent workers, for $B = 0$ (top-left), $B = 0.2$ (top-right), $B = 0.4$ (middle-left), $B = 0.6$ (middle-right), $B = 0.8$ (bottom-left), $B = 1$ (bottom-right).}
\end{center}
On Figure \ref{fig:chap:spatial_numerics:sec:telecommuting:salaries_masses}, we observe a phenomenon similar to what happened in the first simulation related to the sensitivity with respect to $\theta$. Indeed, for $B= 0$, the wages depend on the workplace. Then, as the parameter $B$ increases, people choose to telecommute when
their transport costs are high, see Figure \ref{fig:chap:spatial_numerics:sec:telecommuting:distribution}. As a result, $y_3$ loses some of its positional advantage, since $y_1$ and $y_2$ may hire teleworkers to the right of $y_3$ because the latter do not incur transportation costs.
Progressively, as in Test 1, $y_3$ loses its positional advantage whereas $y_1$ becomes more attractive because its basin of attraction becomes isolated from those of the other two workplaces. As in Test 1,  commuters' residential distributions tend to concentrate in smaller and smaller areas, so that when $B$ is large, no firm has a geographical advantage on the others. Therefore, when $B = 1$, the wages and the number of workers in each workplace are the same, see Figure \ref{fig:chap:spatial_numerics:sec:telecommuting:salaries_masses}. Note that for small values of $\sigma$, the wages of the teleworkers are the same in each workplace, because 
the firms compete for hiring the teleworkers and because the latter do not incur transportation costs. This is what we observe here for $\sigma = 0.1$ where the differences when $B\ge 0.2$ are of the order of $10^{-4}$. 
 The differences observed for $B<0.2$ are due to numerical approximations. 
\begin{center}
\includegraphics{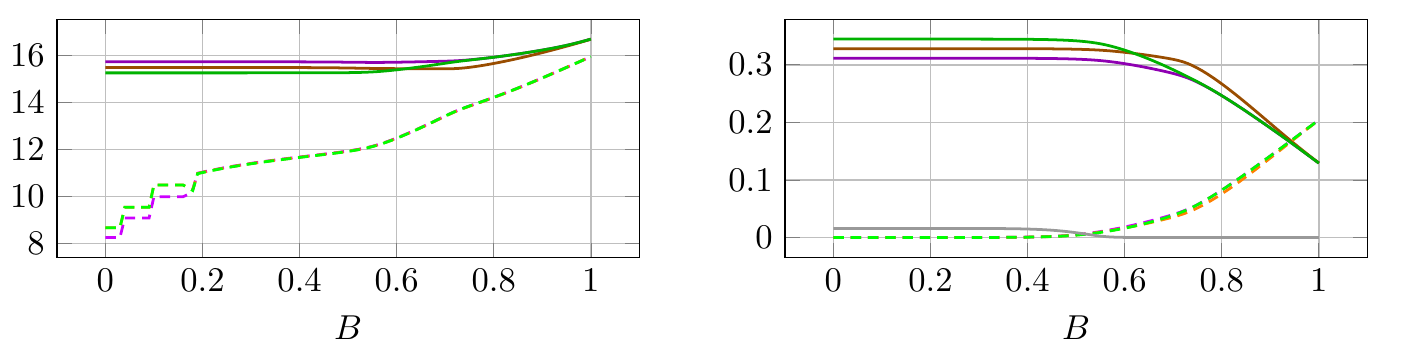}
\captionof{figure}{\label{fig:chap:spatial_numerics:sec:telecommuting:salaries_masses}Wages versus $B$ (on the left) and the number of workers versus $B$ (on the right).}
\end{center}

We also observe on Figure \ref{fig:chap:spatial_numerics:sec:telecommuting:salaries_masses} that commuters are more paid than teleworkers, even if $B = 1$. This comes from the fact that commuters have transport costs. Besides, commuters live in areas where the rental price is higher. Moreover, we observe that when $B$ increases, the wages of both commuters and teleworkers increase. There are two reasons which explain this phenomenon. First, the demand for teleworkers increases with their productivity, and so does their wage. Second, the form of the production function and the fact that $\gamma\in (0,1)$ imply that hiring both types of workers is more productive than hiring just one type.
\begin{center}
\includegraphics{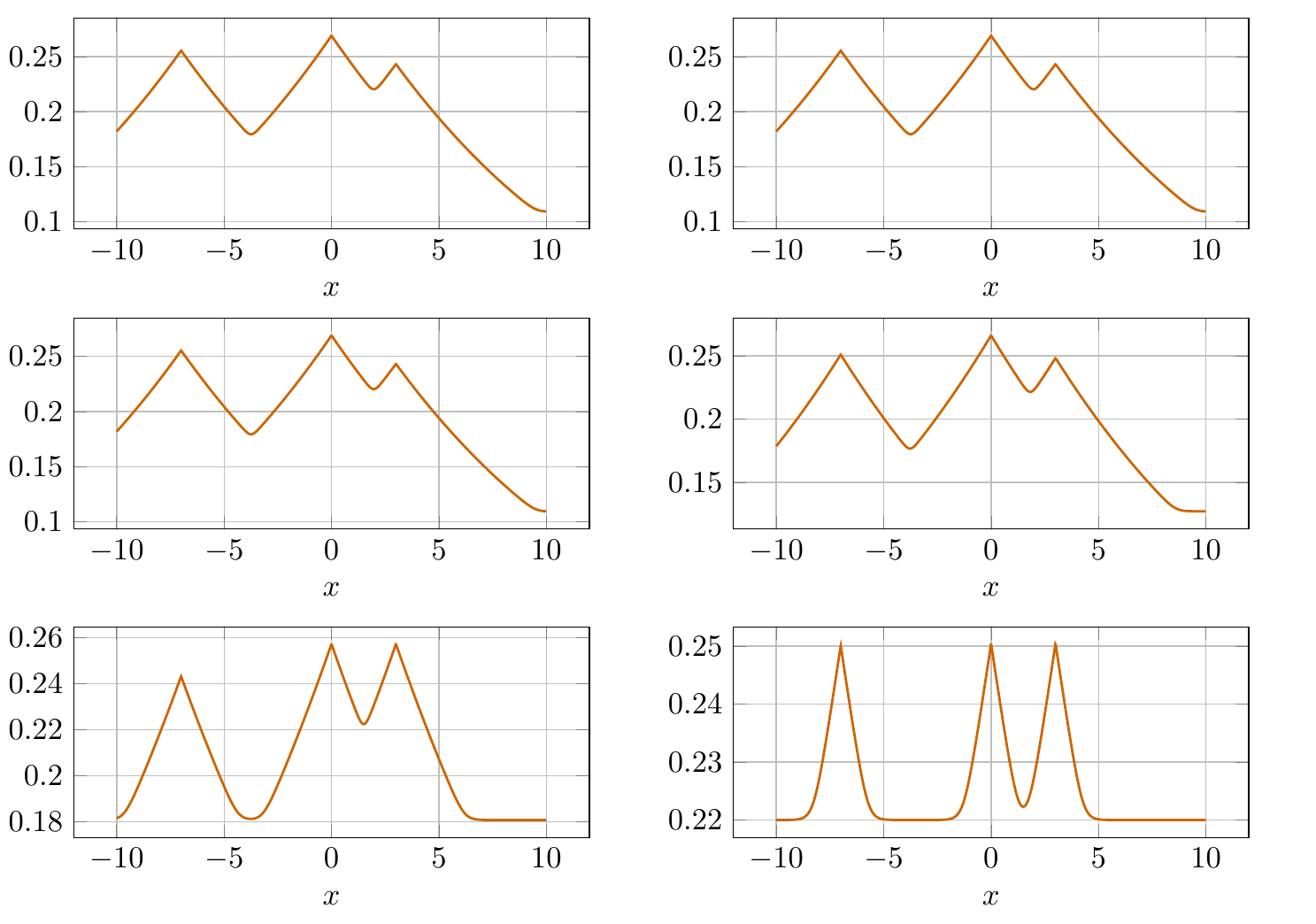}
\captionof{figure}{\label{fig:chap:spatial_numerics:sec:telecommuting:rental_price}Rental price versus $x$, for $B = 0$ (top-left), $B = 0.2$ (top-right), $B = 0.4$ (middle-left), $B = 0.6$ (middle-right), $B = 0.8$ (bottom-left), $B = 1$ (bottom-right).}
\end{center}

On Figure \ref{fig:chap:spatial_numerics:sec:telecommuting:rental_price}, we see that the rental price increases firstly in the area with the highest transport costs.
Then, due to the fact that the geographical position of workplaces is no longer important, the rental price in $y_1$, $y_2$ and $y_3$ is the same. Finally, observe that the rental price in the area occupied by teleworkers is constant. 

In the following paragraph, we present simulations on a two-dimensional domain using the same numerical method.

\subsubsection{A teleworking model in a two-dimensional domain}\label{sim2d}
We aim at extending the teleworking simulation of the previous paragraph to a two-dimensional domain. We obtain similar results with a computation time which increases significantly due to the computational cost of  the integrals in \eqref{equitele} and \eqref{tilmuw} (by a trapezoidal rule adapted to the two-dimensional case). 
\paragraph{Description of the model and the parameters}
We assume that $X = [-10,10]^2$, that the workplaces are located in 
\begin{displaymath}
	y_1 = (-7,7),\quad y_2 = (0,0),\quad y_3 = (3,-3)
\end{displaymath}
and that the production of $y_i$ is given by 
\begin{displaymath}
	\tf_i(l,s) = A^{1-\beta} (l^{\alpha} + Bs^{\alpha})^\frac{\beta}{\alpha} ,\quad \forall (l,s) \in[0,+\infty)^2.
 \end{displaymath}
  In this setting, the transport cost to reach the $i$th workplace, for  commuters living at $x\in X$, is given by 
 \begin{displaymath}
		c_i(x) = \frac{\norm{x - y_i}_2}{2}, 
 \end{displaymath}
 while it is $0$ for  teleworkers.

 As before, we are going to let $B$ vary from 0 to 1.
The parameters used in Test 3 are the same as in Test 2 and are listed in Table \ref{table:chap:spatial_numerics:table:test4} below.

\begin{center}
\begin{tabular}{|c|c|}
\hline Parameter & Value \\ \hline
$\beta$ & 0.7 \\ \hline
$\alpha$ & 0.9 \\ \hline
$A $ & $10^4$ \\ \hline
$w_0 $ & 12 \\ \hline
$\sigma $ & $0.1$ \\ \hline
$\theta $ & $0.7$ \\ \hline
\end{tabular}
\captionof{table}{\label{table:chap:spatial_numerics:table:test4}The parameters used in Test 3.}
\end{center}

\paragraph{Numerical results}

In the figures below, we display the residential distributions of commuters, working at the different workplaces, for different values of $B$.
The color 
(\tikz[baseline=-\the\dimexpr\fontdimen22\textfont2\relax,inner sep=0pt] \draw[orange3,line width=1pt](0,0) -- (5mm,0);), 
(\tikz[baseline=-\the\dimexpr\fontdimen22\textfont2\relax,inner sep=0pt] \draw[bleu3,line width=1pt](0,0) -- (5mm,0);) and
(\tikz[baseline=-\the\dimexpr\fontdimen22\textfont2\relax,inner sep=0pt] \draw[vert3,line width=1pt](0,0) -- (5mm,0);) are associated to the residences of the commuters working at $y_1$, $y_2$ and $y_3$ respectively.
 The color 
(\tikz[baseline=-\the\dimexpr\fontdimen22\textfont2\relax,inner sep=0pt] \draw[bleugris,line width=1pt](0,0) -- (5mm,0);)
is associated to the teleworkers' residences, and 
(\tikz[baseline=-\the\dimexpr\fontdimen22\textfont2\relax,inner sep=0pt] \draw[gris2,line width=1pt](0,0) -- (5mm,0);) 
corresponds to the housing distribution of independent workers.

The analysis is the same as in the latter simulation. When $B$ increases, telecommuting develops in areas with high transport costs, see Figure \ref{fig:chap:spatial_numerics:sec:telecommuting:2D_distribution}. The distributions displayed on the bottom-left of Figure \ref{fig:chap:spatial_numerics:sec:telecommuting:2D_distribution} highlight the same phenomenon of concentration identified in Test 1 and 2, namely the houses of commuters are located in smaller and smaller areas.


\begin{center}
	\includegraphics[scale = 0.85]{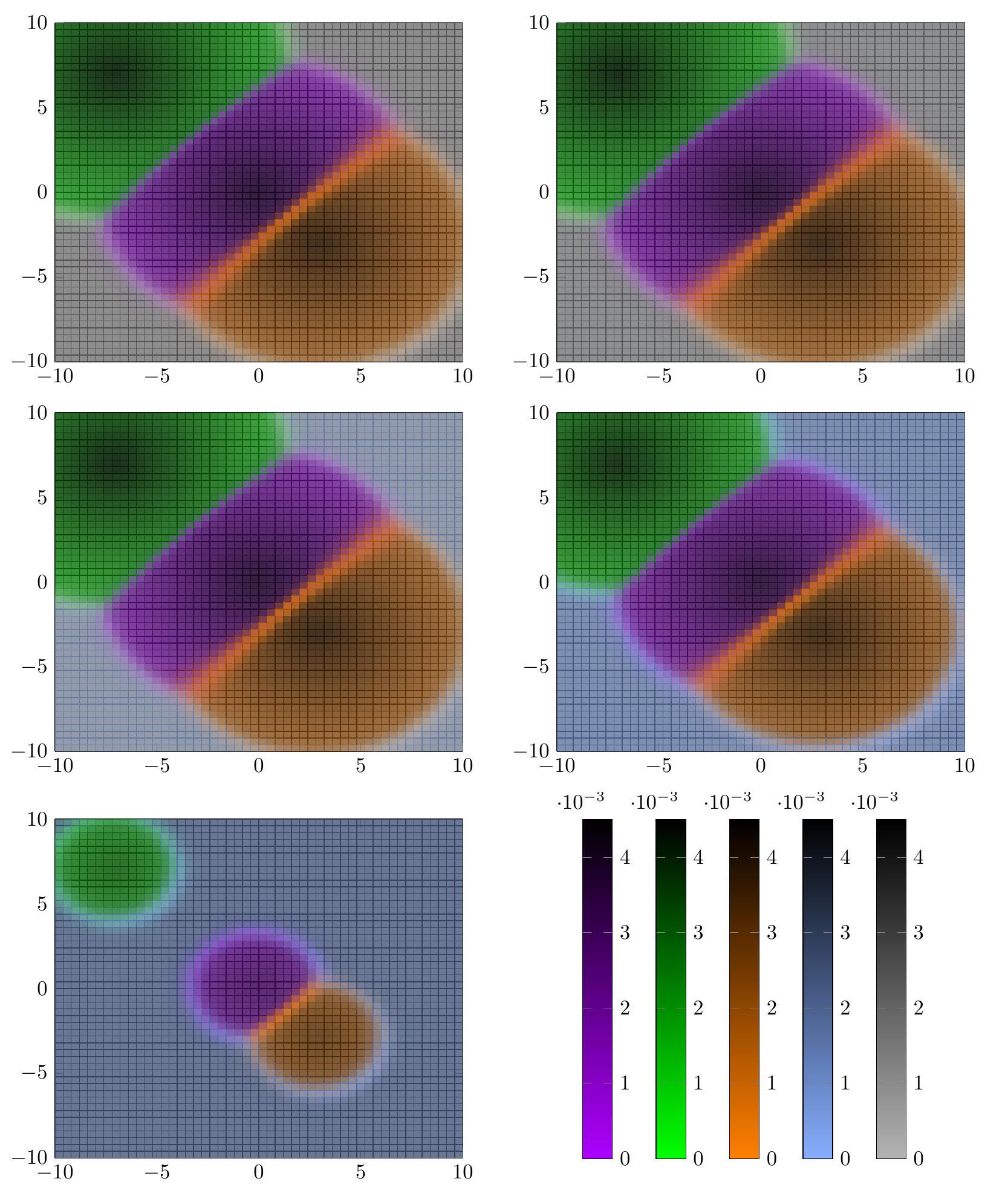}$\;\;\;\;\;\;\;\;\;\;\;\;$
	\captionof{figure}{\label{fig:chap:spatial_numerics:sec:telecommuting:2D_distribution}Distributions of the houses of commuters working for the different workplaces, of teleworkers and of the independent workers, with $B=0$ (top-left), $B = 0.33$ (top-right), $B=0.5$ (middle-left), $B=0.66$ (middle-right), $B=1$ (bottom-left).}
\end{center}

\section*{Appendix}

\subsection*{A. Softmax as the expectation of a max}

Let us first recall that a random variable $\eps$ has a centered standard Gumbel distribution if its cdf has the following double exponential form:
\[\Pro(\eps \le t)=\exp(-\exp(-t-\gamma)), \; \forall t\in \R,   \mbox{ with } \gamma:=-\int_0^{+\infty} \log(s) \exp(-s) \mbox{d}s.\]
Consider now $\eps_0, \cdots, \eps_N$, $N+1$ i.i.d.  distributed with a centered standard Gumbel and for $\beta=\beta_0, \ldots, \beta_N \in \R^{N+1}$ set 
\[V(\beta):=\EE(\max_{i=0, \ldots, N} (\beta_i + \eps_i))\]
since 
\[\Pro( (\max_{i=0, \ldots, N} (\beta_i + \eps_i) \leq t)=\exp \Big(-\Lambda(\beta) \exp(-t-\gamma)\Big), \mbox{ with } \Lambda(\beta)=\sum_{i=0}^N e^{\beta_i}\]
we have
\[\begin{split}
V(\beta )&=\int_{\R}  t  \exp \Big(- \exp(-t +\log\Lambda(\beta)-\gamma)\Big) \exp(-t +\log\Lambda(\beta)-\gamma ) \mbox{d}t \\
& =  \int_{\R}  (s+ \log(\Lambda(\beta)) \exp \Big(- \exp(-s -\gamma)\Big) \exp(-s -\gamma ) \mbox{d}s \\
&= \EE( \eps + \log(\Lambda(\beta)))= \log(\Lambda(\beta))=\log\Big(\sum_{i=0}^N e^{\beta_i} \Big).
\end{split}\]
Recalling the expression \pref{defrsig} for $R_\sigma$, we thus have
\[\EE(\max_{i=0, \ldots, N} (w_i-c_i(x)+\sigma \eps_i))= \sigma \log\Big(\sum_{i=0}^N e^{\frac{w_i-c_i(x)}{\sigma}} \Big) = R_\sigma(x,w)\]
which shows \pref{softmaxmax}.  Moreover, it follows from Lebesgue's dominated convergence theorem that $V$ is differentiable with
\[\frac{\partial V}{\partial \beta_i} (\beta)= \Pro( \beta_i+\eps_i \geq \beta_j +\eps_j, \; \forall j=0, \ldots, N)=\frac{e^{\beta_i}}{\sum_{j=0}^N e^{\beta_j}}\]
which shows formula \pref{probaix}.

\subsection*{B. Existence and uniqueness of equilibria in the teleworking model}
  
  \subsubsection*{Existence}
  
  Under the assumptions of section \ref{sec-tele}, we claim that there exists equilibrium wages i.e. a vector $\tw\in \R_+^{2N}$ solving the fixed-point equation \pref{fpeqrem}. To see this, we first argue as in Lemma \ref{variat} and find that for every $\mu \in \PP(X)$, the functional $\tJ_\mu$ defined in \pref{deftj} admits a unique minimizer $\tw=(w_i^k)_{i=1, \ldots,N, k=1,2}$ which satisfies 
  \[\max_{i,k} w_i^k + \sum_{i=1}^N \tpi_i(w_i^1, w_i^2) \leq \ow:= 2M +  \sum_{i=1}^N \tpi_i(w_0, w_0) +\sigma \log(2N+1)+ w_0\]
  where $M:=\max_i \Vert c_i\Vert_{\infty}$. Thanks to the nonnegativity of $\tpi_i$ and  \pref{tpcoerc0}, we find that the minimizer of $\tJ_\mu$ belongs to $[\uw, \ow]^{2N}$ where $0<\uw\leq \ow$ are bounds that do not depend on $\mu$. The conclusion of Lemma \ref{variat} still holds for $\tJ_\mu$ so that the existence of an equilibrium follows from Brouwer's theorem exactly as in the proof of theorem \ref{existeqsig}. 
  
  \subsubsection*{Uniqueness}
  
  Let us now further assume that the production functions $\tf_i$ are of class $C^2$ on $(0,+\infty)^2$ and that \pref{strongconcaver} holds. Since $-\nabla \tf_i$ is the inverse of $\nabla \tpi_i$, this implies that for every $(w^1, w^2)\in [\uw, \ow]^2$, the $2\times 2$ matrix $D^2 \tpi_i(w^1, w^2)$ is positive definite so that its smallest eigenvalue $\lambda_{\min}(D^2 \tpi_i(w^1, w^2))$ is positive (and depends on $(w^1, w^2)\in [\uw, \ow]^2$ in  a continuous way). Setting $\alpha:=\frac{\theta}{1-\theta}$, and 
  \[\tmu(x, \tw, \alpha):=\frac{ \tR_\sigma(x, \tw)^{\alpha}} {\int_X \tR_\sigma(y, \tw)^{\alpha} \mbox{d} y }, \; \forall (x, \tw, \alpha) \in X\times \R_+^{2N} \times \R_+\]
 we write the system of equilibrium conditions as the system of $2N$ equations in the $2N$ unknown $\tw:=(w_1^1, w_1^2, \ldots w_N^1, w_N^2)$ 
 \[\tG(\tw, \alpha)=0, \mbox{ where } \tG_i^k(\tw, \alpha)=\frac{\partial \tpi_i}{\partial w_i^k}(w_i^1, w_i^2) + \int_X \frac{\partial \tR_{\sigma}}{\partial w_i^k} \tmu(x, \tw, \alpha) \mbox{d}x  \]
 so that the Jacobian matrix (with respect to $\tw$) $\tA$ of $\tG$ reads
 \[\begin{split}
 \tA=&\left(\begin{array}{ccccc} 
 D^2 \tpi_1(w_1^1, w_1^2) & O & O  & \cdots & O \\
 O &  D^2 \tpi_2(w_2^1, w_2^2) & O & \cdots  & O\\
 \cdot & \cdot  & \cdot & \cdots  &  \cdot\\
 \cdot & \cdot  & \cdot & \cdots  &  \cdot\\
 O &   O  & O & \cdots   & D^2 \tpi_N(w_N^1, w_N^2) \\
 \end{array}\right)\\
 &+ \int_X D^2_{\tw \tw} \tR_{\sigma} \tmu \\
& + \int_X \nabla_{\tw} \tR_\sigma  \nabla_{\tw} \tmu^{\top}
 \end{split}\]
  where the matrix in the first line is written in $2\times 2$ block-diagonal form  (and $O$ denotes the zero $2\times 2$ matrix) and the matrix on the second line is positive definite. Since all wages $w_i^k$ belong to $[\uw, \ow] \subset (0,+\infty)$, setting
  \[\nu:=\min_{i=1, \ldots, N} \min_{(w^1, w^2) \in [\uw, \ow]^2} \lambda_{\min}(D^2 \tpi_i(w^1, w^2)) >0,\]
  we have, for every $\txi\in \R^{2N}\setminus \{0\}$
  \[\tA \txi \cdot \txi > \Big( \nu- \int_X  \vert  \nabla_{\tw} \tR_\sigma \vert \;     \vert   \nabla_{\tw} \tmu   \vert      \Big) \vert \txi\vert^2. \]
  Obviously $\vert  \nabla_{\tw} \tR_\sigma \vert\leq \sqrt{2N}$ and assuming that $\alpha \leq 1$, arguing as in the proof of  theorem \ref{thuniq}, we get 
  \[  \int_X \vert   \nabla_{\tw} \tmu   \vert \leq \frac{  \alpha}{w_0} \sqrt{2N}\]
 hence, we deduce that $\tA$ is invertible as soon as
  \[ \alpha \leq   \frac{  w_0 \nu}{ 2N} \]
 For such a choice of $\alpha$, uniqueness of the equilibrium can be shown exactly as in the proof of  theorem \ref{thuniq}.

{\bf Acknowledgments:}  All the authors  were partially supported by the ANR  (Agence Nationale de la Recherche) through MFG project ANR-16-CE40-0015-01.
Y.A. and Q.P. acknowledge partial support from the Chair Finance and Sustainable Development and the FiME Lab  (Institut Europlace de Finance). The paper was completed when Y.A spent a semester at INRIA matherials. G.C. acknowledges the support of the Lagrange Mathematics and Computing Research Center.

\bibliographystyle{plain}

\bibliography{bibli}

\end{document}